\documentclass[DIV=13]{scrartcl}

\usepackage{amssymb, amsthm, mathtools, bbm}
\usepackage[hidelinks, citecolor=blue]{hyperref}
\hypersetup{colorlinks=true, urlcolor = blue, linkcolor=black}
\usepackage[numbers,sort&compress]{natbib}
\newtheorem{theorem}{Theorem}

\newtheorem{lemma}[theorem]{Lemma}

\newtheorem{assumption}[theorem]{Assumption}
\numberwithin{equation}{section}
\numberwithin{theorem}{section}

\newcommand{\rr}{{\mathbb{R}}}
\newcommand{\nn}{{\mathbb{N}}}
\newcommand{\cp}{{\mathbb{C}_+}}
\newcommand{\ee}{{\mathbb{E}\,}}
\newcommand{\pp}{{\mathbb{P}}}
\newcommand{\oh}{{\mathcal{O}}}
\newcommand{\im}{{\operatorname{Im}\,}}
\newcommand{\tr}{{\operatorname{Tr}\,}}
\newcommand{\mean}[1]{{{\langle #1\rangle}}}
\newcommand{\beq}[1]{\begin{equation} \label{#1}}
\newcommand{\eeq}{\end{equation}}
\newcommand{\scal}{\mathcal{S}}
\newcommand{\tcal}{\mathcal{T}}

\begin{document}
\addtokomafont{author}{\raggedright}
\title{\raggedright Random characteristics for Wigner matrices}
\author{\hspace{-.075in}Per von Soosten and Simone Warzel}
\date{\vspace{-.3in}}
\maketitle
\minisec{Abstract} We extend the random characteristics approach to Wigner matrices whose entries are not required to have a normal distribution. As an application, we give a simple and fully dynamical proof of the weak local semicircle law in the bulk.
\bigskip
\bigskip

\section{Introduction} \label{sec:intro}
Starting with the seminal works of Erd\H{o}s, Schlein, and Yau~\cite{MR2481753,MR2810797}, a large portion of recent progress in random matrix theory rests on strong concentration of measure phenomena for the resolvent on almost microscopic scales. Such estimates are an important model-dependent step in proving Wigner-Dyson universality of the local eigenvalue statistics and uniform delocalization bounds for the eigenvectors. The mean-field setting is especially well-understood: the methods in~\cite{MR3068390, MR3800833, MR3699468} give strong results and the central ideas are robust enough to cover a wide range of models (see, for example, ~\cite{Aggarwal2019, MR3770875,MR3602820,MR3962004,MR3183577,MR3230002,MR3941370,nemish} and references therein). The most basic example consists of the $N \times N$ Wigner matrices, whose entries $H_{ij}$ are drawn, independently up to symmetry, from some density with mean zero and variance $N^{-1}$. A typical result proves that the resolvent $G(z) = (H-z)^{-1}$ has essentially deterministic entries
\beq{eq:lscidea} G_{ij}(z) \approx m_{sc}(z) \delta_{ij},
\eeq
the approximation being valid on the smallest possible scale $\im z \gg N^{-1}$. The function
\[m_{sc}(z) = \frac{-z + \sqrt{z^2 - 4}}{2} = \int_{-2}^2 \! \frac{1}{\lambda-z}  \frac{ \sqrt{4-\lambda^2}}{2\pi} \, d\lambda\]
is the Stieltjes transform of the semicircle law.

The approximations~\eqref{eq:lscidea} are usually proved by deriving an approximate self-consistent equation
\beq{eq:selfconsistent} 1 + (z + \mean{G(z)})G(z) \approx 0 ,
\eeq
where we used the notation $\mean{A} = N^{-1} \, \tr A$ that we retain throughout this paper. The stability of the self-consistent equation is then used to show that the approximate solution $G(z)$ is close to the solution $m_{sc}(z)$ of the exact equation. Because it is usually not possible to prove the validity of the self-consistent equation on local scales $\im z \gg N^{-1}$ directly, the analysis uses the stability of~\eqref{eq:selfconsistent} again to show that rough estimates on the validity of~\eqref{eq:selfconsistent} self-improve at finer scales. This idea enables a careful bootstrapping scheme, which allows one to successively improve the scale of the approximation~\eqref{eq:lscidea}.

It was noted by Pastur~\cite{MR0475502}, that the self-consistent equation~\eqref{eq:selfconsistent} can also be viewed as the terminal constraint at $t=1$ of the deterministic advection equation
\beq{eq:pastur} \partial_t G(t,z) = \mean{G(t,z)} \partial_z  G(t,z), \qquad G(0, z) = -z^{-1}, \eeq
which can be easily solved by considering a suitable set of deterministic characteristic curves. In fact, if one generates a Gaussian Wigner matrix dynamically by evolving the entries with Brownian motion~\cite{MR0148397}, one can derive a stochastic version of~\eqref{eq:pastur}
\beq{eq:firstsde} d G(t,z) = \mean{G(t,z)} \partial_z  G(t,z) \, dt + dM(t,z)
\eeq
with some explicit matrix-valued martingale $M(t,z)$ (see, for example,~\cite{MR3835476}). This approach can be used to derive the validity of the semicircle law on global scales $\im z = \oh(1)$ in the infinite-volume limit~\cite{MR2760897}. Deterministic characteristic curves have also featured in the analysis on local scales in more recent random matrix literature such as~\cite{adhikari,bourgade2,Huang2018,MR3405746}.

The works~\cite{uedelocalization1,MR3920502} showed that the SDE~\eqref{eq:firstsde} can also be analyzed on local scales by considering the evolution along the random characteristic
\beq{eq:firstchar} \dot{z}(t) = -\mean{G(t, z(t))},
\eeq
yielding a simple dynamical mechanism for directly proving concentration of measure for the resolvent. The method thus allows one to completely separate any stability arguments from local concentration of measure estimates, thereby circumventing the need for a bootstrap argument.

The purpose of this paper is to show that this approach to local resolvent estimates is not limited to Wigner matrices with Gaussian entries. The basis of this is the construction of a matrix martingale $H(t)$ whose rescaled entries follow a given density $\varrho$ at time $t=1$. This will be possible for densities $\varrho$ satisfying the following assumption. 
\begin{assumption}\label{assump} The density $\varrho > 0$ is strictly positive, has zero mean and unit variance, and the function
\beq{eq:adef} a(h) = \frac{1}{\varrho(h)} \int_h^\infty \! k\varrho(k) \, dk,
\eeq
is bounded and Lipschitz continuous on $\rr$.
\end{assumption}
The integral equation~\eqref{eq:adef} is equivalent to
\[\varrho(h) = \frac{C}{a(h)} \exp\left(-\int_0^h \frac{k}{a(k)} \, dk\right),\]
so the condition that $a$ be bounded essentially amounts to $\varrho$ being sub-Gaussian, whereas the Lipschitz continuity of $a$ is linked to the regularity of $\varrho$. Madan and Yor~\cite{MR1914701} used the Lipschitz continuity of $a$ to construct a scalar martingale $h(t)$ satisfying
\[ d h(t) =  a\left(\frac{h(t)}{\sqrt{t}}\right)^{1/2} \, db(t), \qquad h(0) = 0,\]
where $b$ is a standard Brownian motion. Inspired by an idea of Dupire~\cite{dupire}, they then showed that Kolmogorov's forward equation implies
\beq{eq:tdists} \pp(h(t) \in dx) = t^{-1/2} \varrho(t^{-1/2} x) \, dx.
\eeq
Hence, to construct the matrix process $H(t)$, we take an array $(B_{ij})$ of standard Brownian motions that are independent up to the constraint $B_{ij} = B_{ji}$ and define $H_{ij}(t)$ as the solution of the SDE
\[ d H_{ij}(t) = \frac{1}{\sqrt{N}} \, a\left( \sqrt{\frac{N}{t}} H_{ij}(t)\right)^{1/2} \, dB_{ij}(t), \qquad H_{ij}(0) = 0.\]
Then $ H(t) = \sqrt{t} H $ in distribution, where $H$ is a Wigner matrix whose rescaled entries $\sqrt{N} H_{ij}$ have distribution $\varrho$. It will be important in the sequel that the quadratic variation of $H_{ij}(t)$ satisfies
\[d [H_{ij}](t) = \sigma_{ij}(t) \, dt, \qquad \sigma_{ij}(t) = \frac{1}{N} a\left( \sqrt{\frac{N}{t}} H_{ij}(t)\right).\]

Combining It\^{o}'s lemma with the Neumann expansion of the resolvent yields the matrix-valued SDE
\[dG(t,z) = -G(t,z)\, dH(t) \, G(t,z) + G(t,z) \, dH(t) \, G(t,z) \, dH(t) \, G(t,z)\]
for the resolvent process $G(t,z) = (H(t) - z)^{-1}$ with $ z \in \mathbb{C}_+ $. Multiplying out the drift gives an expression of the form
\beq{eq:resolventsde} dG(t,z) = -G(t,z)\left( dH(t) - \tcal[t, G(t,z)] \, dt \right) G(t,z) + G(t,z) \, \scal [t, G(t,z)] \,G(t,z) \, dt.
\eeq
The matrix-valued operator $\scal[t, \cdot]$, which acts on an $N \times N$ matrix $A$ as
\[\scal[t,A]_{ij} = \delta_{ij} \sum_k \sigma_{ik}(t) A_{kk},\]
is a dominant term whose presence in~\eqref{eq:resolventsde} provides the self-energy corrections in the resolvent. On the other hand, the operator
\[\tcal[t, A]_{ij} =  (1 - \delta_{ij})\sigma_{ij}(t) A_{ji}\]
should be thought of as a finite-volume error reflecting the lack of Hermitian symmetry in our model. Therefore, the self-energy correction coincides with the It\^{o} correction. This is similar in spirit (and identical in the Gaussian case) to the cumulant expansion of He, Knowles, and Rosenthal~\cite{MR3800833}. However, the technical details are simpler here because the quadratic variation process naturally encodes the fluctuations around the Gaussian noise driving the dynamics.

We will illustrate our method by giving a new proof of the weak bulk local semicircle law for the normalized resolvent trace of Wigner matrices whose entries are drawn from densities satisfying Assumption~\ref{assump}. To state this result, we will make use of the stochastic domination language of~\cite{MR3068390}. Let $\{X(u)\}_{u \in U}$ and $\{Y(u)\}_{u \in U}$ be two non-negative $N$-dependent families of random variables. We will say that $X(u)$ is stochastically dominated by $Y(u)$ uniformly in $u \in U$, if, given any $\varepsilon, p > 0$, the inequality
\[ \sup_{u \in U} \pp\left( X(u) > N^\varepsilon Y(u) \right) \le N^{-p}\]
is satisfied for all sufficiently large $N \geq N_0(\varepsilon, p)$. We express this relationship by writing $ X \prec Y $ uniformly in $u \in U$. Although most of the probabilities in this paper can be controlled with an explicit exponential tail, we have refrained from doing so for the sake of brevity. Our proof of the local semicircle law will be valid in a bulk spectral domain
\[D = W + i(\eta, 1)\]
where $W = [W_1, W_2] \subset [-2+\kappa, 2-\kappa]$ for some $\kappa > 0$. For simplicity, we will assume that the minimal spectral scale is given by
\[\eta = N^{-1+\theta}\]
where $\theta > 0$ is fixed, but arbitrarily small.

\begin{theorem}\label{thm:lsc} We have
\[ \sup_{z \in D} \left| \mean{G(1, z)} - m_{sc}(z) \right| \prec \frac{1}{\sqrt{N\eta}}.\]
\end{theorem}

We stress that this theorem first appeared in~\cite{MR2481753}. It has been extended both to Wigner matrices with minimal assumption on the distribution of the entries~\cite{Aggarwal2019} and to random matrices with a more general spread-out variance profile~\cite{MR3068390}.

Theorem~\ref{thm:lsc} implies that $\mean{ \im G(1,z)}$ is bounded both above and away from zero when $z \in D$. The Schur complement formula and classical concentration of measure results show that
\[ \sup_{z \in D} \left( \frac{1 + m_{sc}(z)}{N\eta}\right)^{- 1/2} \left| \frac{1}{G_{kk}(1, z)} +z + m_{sc}(z)\right| \prec   1 \]
so Taylor expansion yields
\beq{eq:diagentries}\sup_{z \in D} \left| G(1, z)_{kk}  - m_{sc}(z)\right| \prec \frac{1}{\sqrt{N\eta}}.
\eeq
Similar considerations apply to the off-diagonal entries
\beq{eq:offdiagentries} \sup_{z \in D} \left| G(1, z)_{jk}\right| \prec \frac{1}{\sqrt{N\eta}}, \qquad j\neq k.
\eeq
These calculations have become standard and are explained in great detail in~\cite{MR3699468}. However, given Theorem~\ref{thm:lsc}, no further bootstrapping is required to conclude~\eqref{eq:diagentries} and~\eqref{eq:offdiagentries}.

The organization of this paper is as follows. In Section~\ref{sec:selfenergy}, we show in which sense the evolution~\eqref{eq:resolventsde} approximates~\eqref{eq:firstsde}. Then, in Section~\ref{sec:characteristics}, we prove that~\eqref{eq:firstchar} defines an approximate characteristic flow, which we use to derive the local semicircle law in Section~\ref{sec:lsc}.

\section{The self-energy correction}\label{sec:selfenergy}
In comparison the Gaussian case, the main complication in the random characteristic approach for more general Wigner matrices is that the self-energy operator $\scal[t, \cdot]$ remains random and time-dependent. Nevertheless, we will prove in Theorem~\ref{thm:hoeffding} that
\[\scal[t, G(t,z)] \approx \mean{G(t,z)} + \oh\left(\sqrt{\frac{ 1 + \im \mean{G(t, z)}}{N \, \im z}} \right)\]
with very high probability so that the characteristic curve~\eqref{eq:firstchar} still counteracts a large part of the term
\[G(t,z) \, \scal [t, G(t,z)] \,G(t,z) \, dt\]
in~\eqref{eq:resolventsde}. The resulting error term is not small a-priori, but retains a specific structure that enables the Gr\"{o}nwall scheme of Lemma~\ref{thm:counterterm}.

For technical reasons, we will prove Theorem~\ref{thm:hoeffding} for spectral parameters $z$ in a very large domain
\beq{eq:dprimedef}D^\prime = (W_1 - 3\eta^{-1}, W_2 + 3\eta^{-1}) + i (\eta/4, 1 + 3\eta^{-1})
\eeq
and for times $t$ that are greater than a cutoff
\beq{eq:t0def} t_0 = N^{-K}
\eeq
with some fixed $K \in \nn$ to be specified later in the proof of Lemma~\ref{thm:counterterm}. In the statement of Theorem~\ref{thm:hoeffding}, and throughout this paper, $\|\cdot\|$ denotes the operator norm.

\begin{theorem}\label{thm:hoeffding} For any choice of $ K \in \nn $ in~\eqref{eq:t0def} we have
\[\sup_{t \in [t_0, 1], z \in D^\prime} \left(\frac{ 1 + \im \mean{G(t, z)}}{N \, \im z} \right)^{-1/2} \| \scal[t, G(t,z)] - \mean{G(t, z)} \| \prec 1.\]
\end{theorem}
\begin{proof}
We first show that
\beq{eq:singlecounter} \left |\scal[t, G(t,z)]_{kk} - \mean{G(t,z)} \right| \prec \sqrt{\frac{ 1 + \im \mean{G(t, z)}}{N \, \im z}} 
\eeq
uniformly in  $  t \in [t_0, 1] $, $ z \in D^\prime $, and $k \in \{1, \dots, N\}$. Let $H^k(t)$ denote the matrix obtained by replacing the $k$-th row and column of $H(t)$ by zeros. Denoting by $G^k(t,z)$ the resolvent of $H^k(t)$, the resolvent identity implies
\beq{eq:minorid}G_{jj}(t,z) = G^k_{jj}(t,z) + \frac{G_{kj}(t,z) G_{jk}(t,z)}{G_{kk}(t,z)}
\eeq
when $j \neq k$. Since $\sigma_{jk}(t) \leq CN^{-1} $ by Assumption~\ref{assump}, we conclude the uniform deterministic estimate
\begin{align*}
 \left|\sum_j \sigma_{kj}(t) (G^k_{jj}(t,z) - G_{jj}(t,z))\right|  & \leq \frac{C}{N \, \im z} + \frac{1}{|G_{kk}(t,z)|}\sum_j \sigma_{kj}(t) |G_{kj}(t,z)|^2 \\
& \leq \frac{C}{N \, \im z} \left( 1+ \frac{\im G_{kk}(t,z)}{|G_{kk}(t,z)|} \right)  \leq  \frac{C}{N \, \im z}  
\end{align*}
using~\eqref{eq:minorid} and the trivial resolvent bound for the summand with $j=k$. Similarly, we have 
\[\left| \langle G^k(t,z) \rangle - \langle G(t,z) \rangle \right| \leq \frac{C}{N\, \im z}.\]
To prove~\eqref{eq:singlecounter} it thus suffices to show that
\beq{eq:singlecounter1} \left| \sum_j \sigma_{kj}(t) G^k_{jj}(t,z) - \langle G^k(t,z) \rangle \right| \prec  \sqrt{\frac{1+ \im \langle G^k(t,z) \rangle}{N \, \im z}} 
\eeq
uniformly in $  t \in [t_0, 1] $, $ z \in D^\prime $, and $k \in \{1, \dots, N\}$. After conditioning on $H^{k}(t)$, the random variables $\sigma_{kj}(t) G^k_{jj}(t,z) $ are independent and bounded by $ C\, N^{-1} |G^k_{jj}(t,z)| $. Using Hoeffding's inequality with respect to the conditional probability $\pp_k$ we get
\[\pp_k\left( \left| \sum_j \sigma_{kj}(t) G^k_{jj}(t,z)  - \mu \right|  >  \upsilon \alpha  \right) \le 2 \exp \left(- c\alpha^2 \right)\]
with
\[\mu =  \ee_k \left[ \sum_j \sigma_{kj}(t) G^k_{jj}(t,z)\right]  = \mean{ G^k(t,z)}\]
and
\[\upsilon^2 = \frac{1}{N^2}\sum_{j} |G^k_{jj}(t,z)|^2 \le \frac{1}{N^2} \sum_{i,j} |G^k_{ij}(t,z)|^2 = \frac{ \mean{\im G^k(t,z)}}{N \, \im z},\]
which proves~\eqref{eq:singlecounter1}.

The extension of~\eqref{eq:singlecounter} to the maximum over $k$ is by the union bound, whereas the extension to the supremum over all $z \in D^\prime$ and $t \in [t_0, 1]$ beyond the cutoff $ t_0 = N^{-K }$ is by a stochastic continuity bound that we turn to in the next step. We fix $ r > 0 $ and consider the neighborhood
\[B_r(t,z) = \{(s, w) \in  [t_0, 1] \times D^\prime: |t-s|^2 + |z-w|^2 \le r^2 \}\]
around some point $(t,z) \in [t_0, 1] \times D^\prime$. 
The theorem then follows from~\eqref{eq:singlecounter} and 
\beq{eq:stochcont} \sup_{(s,w) \in B_r(t,z)} \|\scal[s, G(s, w)] - \scal[t, G(t, z)] \| \prec N^{K} r^{1/2}.
\eeq
Indeed, for every $L \in \nn$ and $r = N^{-L}$, there exists an $L^\prime \in \nn$ and a finite grid  $\Lambda \subset [t_0, 1] \times D^\prime$ of cardinality $ |\Lambda| \le N^{L^\prime}$ such that $\operatorname{dist}(\Lambda, [t_0,1] \times D) \le r $.
Let $\varepsilon, p > 0$ be arbitrary. Applying the union bound to~\eqref{eq:singlecounter} and~\eqref{eq:stochcont} shows that the events
\[ \| \scal[t, G(t,z)] - \mean{G(t, z)} \|   \le N^\varepsilon  \left(\frac{ 1 + \im \mean{G(t, z)}}{N \, \im z} \right)^{1/2}\]
and
 \[\sup_{(s,w) \in B_r(t,z)} \|\scal[s, G(s, w)] - \scal[t, G(t, z)] \|   \le N^{\varepsilon + K} r^{1/2} \]
hold simultaneously for all $(t,z) \in \Lambda$ with probability $1 - N^{-p}$, when $N$ is large enough. The claim of Theorem~\ref{thm:hoeffding} follows since $L$ can be chosen arbitrarily large.

Finally, to prove~\eqref{eq:stochcont}, we note first that $\|H(t+s) - H(t)\|$ is a submartingale in $s \geq 0$. Doob's inequality and classical norm bounds~\cite{MR2906465} for the rescaled Wigner matrix $H(t+r) - H(t)$ show that
\[ \left(\ee \left| \sup_{s \in [t, t+r]} \|H(s) - H(t)\| \right|^p \right)^{1/p}  \le \frac{p}{p-1}  \left( \ee \|H(t+r) - H(t)\|^p\right)^{1/p} \le C_p \, r^{1/2}  \]
 for all $ p > 1 $ with some $p$-dependent constant $ C_p < \infty$. Since the resolvent is Lipschitz continuous with constant $\eta^{-2}$ in both $H$ and $z$, this implies
\[\sup_{|t-s|^2 + |z-w|^2 \le  r^2}  \|G(s, w) - G(t,z)\| \prec \eta^{-2} r^{1/2}.\]
Doob's inequality also implies that
\[\left( \ee  \left| \sup_{s \in [t, t+r]} |H_{kj}(s) - H_{kj}(t) \right|^p \right)^{1/p} \leq C_p \sqrt{\frac{r}{N}}\]
for any pair of indices $ j,k $. In combination with the Lipschitz continuity of $a$, this implies
\[\sup_{s \in [t, t+r]} |\sigma_{kj}(s) - \sigma_{kj}(t)| \le \sup_{s \in [t, t+r]}  \frac{C}{\sqrt{N}} \left| \frac{H_{kj}(s)}{\sqrt{s}} -  \frac{H_{kj}(t)}{\sqrt{t}}\right|  \prec  \frac{r^{1/2}}{N t_0} = N^{K-1} r^{1/2} \]
for any $t \geq t_0 $. Since $\scal[t, G(t,z)]$ is a polynomial combination of the $\sigma_{kj}(t)$ and $G(t,z)$, which are bounded by $ N^{-1} $ and $ \eta^{-1} \le N $ respectively, the bound~\eqref{eq:stochcont} follows.
\end{proof}

\section{Fluctuations along characteristic curves} \label{sec:characteristics}
We will show that, for fixed realizations of the randomness, the unique solutions of
\[\dot{\gamma}(t,z) = -\mean{G(t, \gamma(t,z))}, \qquad \gamma(0,z) = z\]
serve as approximate characteristic curves of the resolvent SDE~\eqref{eq:resolventsde}. 
We start these curves at spectral parameters $z$ in an initial spectral domain
\[ D_0 = (W_1 - 2\eta^{-1}, W_2 + 2\eta^{-1}) + i(\eta/2, 1+2\eta^{-1}) \setminus B_\delta(0),\]
where $\delta > 0$ is any constant satisfying
\beq{eq:deltadef} \frac{1}{2\delta} - 2\delta \geq 1 + \sup \{|z|: z \in D\}.
\eeq
Thus $G(0, z) = -z^{-1}$ is uniformly bounded and Lipschitz continuous in $D_0$. Given an initial point $z \in D_0$, we consider the process
\[R(t,z) = G(t \wedge \tau_z, \xi(t,z))\]
where
\[\xi(t,z) = \gamma(t \wedge \tau_z, z)\]
is the characteristic flow stopped at
\[\tau_z = \inf \{ t > 0: \im \xi(t,z) \le \eta/4\}.\]
The main observation is that $\mean{R(t,z)}$ is approximately constant.

\begin{theorem}\label{thm:maincharbound} The process $R(t,z)$ satisfies
\[ \sup_{t \le 1} \left| \mean{R(t,z)} - \mean{R(0, z)} \right| \prec \frac{1}{\sqrt{N \eta}}  \]
uniformly in $z \in D_0$.
\end{theorem}
\begin{proof}
Since $\xi$ defines a piecewise $C^1$-process, an application of It\^{o}'s lemma shows that $R(t,z)$ satisfies the same SDE as $G(t,z)$ but with an additional counter-term in the drift. More precisely, while $t \le \tau_z$,  the evolution consists of two terms
\[dR(t,z) = dF(t,z) + dA(t,z)\]
with
\begin{align*} dF(t,z) \ & = -R(t,z) \left( dH(t) - \tcal[t, R(t,z)] \, dt \right) R(t,z) , \\
dA(t,z) \ & = R(t,z) \left(\scal [t, R(t,z)] - \mean{R(t,z)} \right) R(t,z) \, dt .
\end{align*}
When dealing with the integrated versions of these processes, we will choose the initial conditions $F(0,z) = A(0,z) = 0$ so that $ \mean{R(t,z)} - \mean{R(0, z)} = \mean{F(t,z)} + \mean{A(t,z)} $. The proof of Theorem~\ref{thm:maincharbound} is then a direct consequence of the estimates on $F(t,z)$ and $A(t,z)$ provided in Lemma~\ref{thm:mgbounds} and Lemma~\ref{thm:counterterm}, respectively. 
\end{proof}

The proofs of the subsequent results will repeatedly use the crucial fact that any continuously differentiable function $f$ satisfies the identity
\begin{align*}\int_a^b \! f^\prime(\im \xi(s, z)) \, \mean{\im R(s,z)} \, ds &= -\int_a^b \! f^\prime(\im \xi(s, z)) \, d \left(\im \xi(s, z) \right)\\
& = f(\im \xi(a, z))-f(\im \xi(b, z))
\end{align*}
provided that $a, b \le \tau_z$. The term $f(\im \xi(a, z))-f(\im \xi(b, z))$ is then usually estimated by a trivial bound in terms of~$\eta$. We refer to these two steps simply as the ``integration trick''.

\begin{lemma}\label{thm:mgbounds} The process $F(t,z)$ satisfies
\[ \sup_{t \le 1} |\mean{F(t,z)}| \prec \frac{1}{N\eta} \]
uniformly in $z \in D_0$.
\end{lemma}
\begin{proof} The martingale part of $F(t,z)$ is
\[dM(t) = -R(t,z) \, dH(t) R(t,z)\]
and the quadratic variation of its unit trace $\mean{M}$ is given by
\[d\left[ \mean{M} \right](t) =  {\mean{R^\ast(t,z) \, dH(t) \, R^\ast(t,z)}} \, \mean{R(t,z) \, dH(t) \, R(t,z)} \le \frac{C}{N^2} \|\sqrt{\sigma}(t) \odot R^2(t,z)\|_2^2 \, dt,\]
where $ \| \cdot \|_2 = \sqrt{ \tr | \cdot |^2 } $ denotes the Hilbert-Schmidt norm, $\odot$ denotes the entrywise product, and $ \sqrt{\sigma}(t) $ is the matrix with entries $ \sqrt{\sigma_{jk}(t) } $.  If $t \le \tau_z$, we conclude that
\begin{align*} 
[\mean{M}](t) & \le \frac{C}{N^2} \int_0^t \!  \|\sqrt{\sigma}(s) \odot R^2(s,z)\|_2^2 \, ds\\
&\le \frac{C}{N^3}  \int_0^t \!  \frac{ \tr |R|^2(s,z)}{ (\im \xi(s,z))^2}  \, ds\\
&\le  \frac{C}{N^2} \int_0^t \!  \frac{\mean{\im R(s,z)}}{(\im \xi(s, z))^3} \, ds = \frac{C}{N^2} \left(\frac{1}{ (\im \xi(t, z))^2} -  \frac{1}{ (\im \xi(0, z))^2} \right)
\end{align*}
by the integration trick. Thus
\[\sup_{t \le \tau_z} |\mean{M(t)}| \prec (N\eta)^{-1}\]
by the Burkholder-Davis-Gundy inequality and Markov's inequality. The other term in $F(t,z)$,
\[\sup_{t \le \tau_z} \left|  \int_0^t \!  \mean{ R(s,z) \tcal[s, R(s,z)] R(s,z)} \, ds \right|  \le \frac{C}{N} \int_0^{\tau_z} \! \frac{  \mean{|R(s,z)|^2}}{\im \xi(s,z)} \, ds =  \frac{C}{N} \int_0^{\tau_z} \!  \frac{\mean{\im R(s)}}{(\im \xi(s, z))^2} \, ds,\]
is also stochastically dominated by $(N\eta)^{-1}$ uniformly in $ z \in D_0$ because of the integration trick.
\end{proof}

\begin{lemma}\label{thm:counterterm} The process $A(t,z)$ satisfies
\[ \sup_{t \le 1} | \mean{A(t,z)}| \prec \frac{1}{\sqrt{N\eta}}\]
uniformly in $z \in D_0$.
\end{lemma}
\begin{proof} We will split the integral
\beq{eq:intsplit} | \mean{A(t,z)}| \le \left( \int_0^{t_0} + \int_{t_0}^t  \right) \! \left| \mean{ R(s,z) \left(\scal [s, R(s,z)] - \mean{R(s,z)} \right) R(s,z) } \right|\, ds
\eeq
at the point $t_0$ from~\eqref{eq:t0def}. By the trivial bound on the resolvent, the characteristic $\xi(t,z)$ started at $z \in D_0$ remains in the domain $D^\prime$ from~\eqref{eq:dprimedef} for all $t \le 1$. Setting
\[u(t) = 1 + \im \mean{R(t, z)},\]
Theorem~\ref{thm:hoeffding} shows that the second part of the integral in~\eqref{eq:intsplit} is bounded by 
\beq{eq:cbound} N^{\theta/2}   \int_{t_0}^t \!  \frac{ \mean{ | R(s,z)|^2 } } {\sqrt{N \, \im \xi(s,z)}}  \sqrt{u(s)}\, ds 
  \le    \frac{ N^{\theta/2}}{\sqrt{N} } \int_{0}^t \! \frac{\im \mean{ R(s,z)}}{(\im \xi(s,z))^{3/2}} \, u(s)\, ds
\eeq
with probability $ 1 - N^{-p} $ for arbitrary $p > 0$ and large enough $N$. The first part of the integral in~\eqref{eq:intsplit} can be bounded by $C t_0  \eta^{-3}$ using the trivial bound on the operator norm of the resolvent. So choosing the cutoff exponent $K$ in~\eqref{eq:t0def} large enough and using Lemma~\ref{thm:mgbounds} yields
\[ u(t)  \leq 1 + u(0)  + N^{\frac{\theta}{2}}  \int_{0}^t \! \frac{\im \mean{ R(s,z)}}{(\im \xi(s,z))^{3/2}} \, u(s)\, ds\]
on an event that also has probability $ 1 - N^{-p} $ for arbitrary $p > 0$ and large enough $N$. On this event, Gr\"{o}nwall's inequality implies 
\[u(t)  \leq \left( 1 + u(0) \right) \exp\left(N^{\frac{\theta}{2}}  \int_0^t \!  \frac{\im \mean{R(s, z)}}{N^{1/2} \, (\im \xi(s,z))^{3/2}}\, ds \right)\]
and the integral inside the exponential is bounded by $ 4 (N\eta)^{-1/2} = 4N^{-\theta/2}$ because of the integration trick. Since $\mean{R(0,z)} \le \delta^{-1}$ for $z \in D_0$, we have shown that
\[\sup_{t \le 1} u(t) \prec 1.\]
We now insert this bound on $u$ back into the integral in~\eqref{eq:cbound}.  Choosing the cutoff exponent $K$ in~\eqref{eq:t0def} large enough yields
 \[ \sup_{t \le 1} | \mean{A(t,z)}| \prec \frac{1}{N\eta} +  \frac{1}{\sqrt{N}} \int_{0}^{1} \! \frac{\im \mean{ R(s,z)}}{(\im \xi(s,z))^{3/2}} \, ds \prec \frac{1}{\sqrt{N\eta}}\]
via the integration trick.
\end{proof}

Having established Theorem~\ref{thm:maincharbound}, we now argue that the stochastic domination holds simultaneously for a continuum of points using a discretization argument.

\begin{theorem}\label{thm:finaldom} Let $\tau_z^\prime = \inf \{t > 0: \im \gamma(t,z)  \le \eta/2\} $.
Then
\beq{eq:finaldom} \sup_{z \in D_0 } \sup_{t \le 1 \wedge \tau_z^\prime} |\mean{R(t, z)} - \mean{R(0,z)}| \prec \frac{1}{\sqrt{N \eta}}.
\eeq
\end{theorem}
\begin{proof} 
We begin by proving a Lipschitz bound for the characteristic flow  $\gamma(t,\cdot)$.
Fix $z, w \in \cp$ and let
\[u(t) = \gamma(t, z) - \gamma(t, w).\]
As long as $ \im \gamma(t, z) , \im \gamma(t, w)> \eta/4$, we have
\begin{align*} |u(t)|  &\le |u(0)| +  \int_0^t \! |  \mean{ R(s, z) - R(s,w)} | \,ds\\
&\le   |u(0)| + \frac{1}{N}  \int_0^t \! |u(s)|  \|R(s, z)\|_2 \| R(s,w)\|_2 \,ds\\
&\le |u(0)| +   \int_0^t \! \frac{ |u(s)| }{2}  \left( \frac{\mean{\im R(s, z)}}{\im\gamma(s,z)} +  \frac{\mean{\im R(s, w)}}{\im\gamma(s,w)} \right) \,ds
\end{align*}
so Gr\"{o}nwall's inequality and the integration trick yield
\[\log \frac{|u(t)| }{|u(0)| } \le   \frac{1}{2}  \int_0^t \! \left( \frac{\mean{\im R(s, z)}}{\im \gamma(s,z)} +  \frac{\mean{\im R(s, w)}}{\im \gamma(s,w)} \right) \,ds = \frac{1}{2} \log \left(\frac{\im z \, \im w } { \im \gamma(t,z)\,   \im \gamma(t, w)}\right). \]
Rearranging, we get
\beq{eq:gronwall}\left| \gamma(t,z) - \gamma(t,w) \right| \leq \sqrt{\frac{\im z \, \im w } { \im \gamma(t,z)\,   \im \gamma(t, w)}} \, | z-w |.\eeq

To prove the theorem, we choose a finite grid $\Lambda \subset D_0$ such that its cardinality is bounded by $ |\Lambda| \le N^L$ for some $L \in \nn$ and
\[\operatorname{dist}(D_0, \Lambda) \le \frac{\eta^4}{\sqrt{N\eta}}.\]
From Theorem~\ref{thm:maincharbound} and the union bound, we have
\beq{eq:initialdom} \sup_{z \in \Lambda } \sup_{t \le 1} |\mean{R(t, z)} - \mean{R(0,z)}| \prec \frac{1}{\sqrt{N \eta}},
\eeq
so it suffices to show that the left side of~\eqref{eq:initialdom} controls the left side of~\eqref{eq:finaldom}. Given $z \in D_0$, we pick $w \in \Lambda$ such that $|z - w| \le \frac{\eta^3}{\sqrt{N\eta}}$. Then $\tau_w^\prime \le \tau_w$  and $\tau_{z}^\prime \le \tau_w$ since~\eqref{eq:gronwall} guarantees that
\[ |\gamma(t,z) - \gamma(t, w)| \le C \frac{\eta^{2}}{\sqrt{N\eta}}\]
as long as $\im \gamma(t,z) \geq \eta/4$ and  $\im \gamma(t,w) \geq \eta/4$. The trivial $C\eta^{-2}$-Lipschitz continuity of the resolvent in $D_0$ then shows that the process $\mean{R(t, z)}$ stays within an error $C/\sqrt{N\eta}$ of $\mean{R(t,w)}$ for all times $t \le \tau_z^\prime$.
\end{proof}

\section{Proof of the local semicircle law}\label{sec:lsc}
To prove Theorem~\ref{thm:lsc} we choose an arbitrarily small $ \varepsilon > 0 $ and prove that the local semicircle law is valid on the event
\[\mathcal{A}_\varepsilon = \left\{  \sup_{z \in D_0 } \sup_{t \le 1 \wedge \tau_z^\prime} |\mean{R(t, z)} - \mean{R(0,z)}| \le \frac{N^\varepsilon}{\sqrt{N \eta}} \right\}\]
with an error that is also of order $N^\varepsilon/\sqrt{N\eta}$. By Theorem~\ref{thm:finaldom} the event $\mathcal{A}_\varepsilon$ has probability $ 1 - N^{-p} $ for arbitrary $ p > 0 $ when $N$ is large enough. We begin by noting that the characteristic flow may be computed explicitly on $\mathcal{A}_\varepsilon$.

\begin{lemma}\label{thm:charcontrol} Let $ \varepsilon > 0 $ and suppose $\mathcal{A}_\varepsilon$ occurs. Then for every $z \in D$, there exists $w \in D_0$ such that $z = \gamma(1, w)$ and
\[ \left| w + \frac{1}{w} - z \right| \le \frac{CN^\varepsilon}{\sqrt{N\eta}}.\]
\end{lemma}
\begin{proof}
Let $\lambda$ be the time-reversal of $\gamma$ defined by
\[\dot{\lambda}(t, \zeta) = \mean{G(1-t, \lambda(t, \zeta))}, \qquad \lambda(0, \zeta) = \zeta\]
so that $\gamma(1, \lambda(1,\zeta)) = \zeta$. Whenever $\zeta \in D$ and $\lambda(1, \zeta) \in D_0$, we necessarily have $\tau_\zeta > 1$, so 
\beq{eq:charformula} 
\gamma(1, \zeta)  = \zeta - \mean{R(0,\zeta)} -  \int_0^1 \!\left(  \mean{R(s,\zeta)} - \mean{R(0,\zeta)}\right)  \, ds = \zeta + \zeta^{-1} + \oh\left(\frac{N^\varepsilon}{\sqrt{N\eta}}\right)
\eeq
by the definition of $\mathcal{A}_\varepsilon$. The desired $w$ in the conclusion of the lemma will be
\[w = \lambda(1, z),\]
which satisfies $ w \in D_0 \cup B_\delta(0) $
by the trivial resolvent bound. It thus remains to prove that $w \notin B_\delta(0)$. For this, we note that $\lambda(1, D)$ is simply connected since $\lambda(1, \cdot)$ is a homeomorphism. If there were any point $\zeta \in \lambda(1, D) \cap B_\delta(0)$, there would also be some point $\zeta^\prime \in \lambda(1, D) \cap D_0$ with $|\zeta^\prime| = 2\delta$. Since $\zeta^\prime \in D_0$, the relation~\eqref{eq:charformula} implies
\[ | \gamma(1, \zeta^\prime) | + \oh\left(\frac{N^\varepsilon}{\sqrt{N\eta}}\right) \geq \left|\zeta^\prime   + \frac{1}{\zeta^\prime} \right| \geq \frac{1}{2\delta} - 2 \delta,  \] 
which leads to the contradiction $\gamma(1, \zeta^\prime) \notin D$ by~\eqref{eq:deltadef}.
\end{proof}

Before proving Theorem~\ref{thm:lsc}, we mention that the relation
\beq{eq:semicircleeq}w^\prime + \frac{1}{w^\prime} = z, \qquad w^\prime \in \cp
\eeq
is equivalent to $-1/w^\prime = m_{sc}(z)$. Since the semicircle law is analytic in the bulk interval $W$, its Stieltjes transform $m_{sc}$ is Lipschitz continuous in $D$ with a constant independent of $N$.

\begin{proof}[Proof of Theorem~\ref{thm:lsc}]
It suffices to prove that
\[  \sup_{z\in D} \left| \mean{G(1, z)} - m_{sc}(z) \right| \leq \oh\left(\frac{N^\varepsilon}{\sqrt{N\eta}} \right) \]
on the event $\mathcal{A}_\varepsilon$ for all $\varepsilon <  \theta/2 $. Let $z \in D$, let $w = w(z) \in D_0$ be the point furnished by Lemma~\ref{thm:charcontrol}, and let $w^\prime = - m_{sc}(z)^{-1} $ be the solution of \eqref{eq:semicircleeq}. By Lemma~\ref{thm:charcontrol} we have $ w + w^{-1} \in D $ for all sufficiently large $ N $, so the Lipschitz continuous dependence of $1/w^\prime$ on $z \in D$ implies
\[ \left|\frac{1}{w} - \frac{1}{w^\prime} \right| = \oh\left(\frac{N^\varepsilon}{\sqrt{N\eta}} \right). \]
On the event $\mathcal{A}_\varepsilon$ the bound
\begin{align*}
\left| \mean{G(1, z)} - m_{sc}(z) \right| \ &  \leq \left| \mean{G(1, z)} -  \mean{G(0,w)} \right| +  \left|\frac{1}{w} - \frac{1}{w^\prime} \right| \\
& \leq  \sup_{t \le 1 \wedge \tau_w^\prime} |\mean{R(t, w)} - \mean{R(0,w)}| + \oh\left(\frac{N^\varepsilon}{\sqrt{N\eta}} \right)
 \leq \oh\left(\frac{N^\varepsilon}{\sqrt{N\eta}} \right) 
\end{align*}
is valid with a uniform constant for all $z \in D$.
\end{proof}

\minisec{Acknowledgments}
We are grateful to two anonymous referees for their helpful suggestions. We would also like to thank Institut Mittag-Leffler for their hospitality during the early stages of this project. This work was supported by the DFG grants SO 1724/1-1 (P.\,S.) and EXC-2111 -- 390814868 (S.\,W.).

\bibliographystyle{abbrv}
\bibliography{References}
\bigskip
\bigskip
\begin{minipage}{0.5\linewidth}
\noindent Per von Soosten\\
Department of Mathematics\\
Harvard University\\
\verb+vonsoosten@math.harvard.edu+ 
\end{minipage}%
\begin{minipage}{0.5\linewidth}
\noindent Simone Warzel\\
MCQST \& Zentrum Mathematik \\
Technische Universit\"{a}t M\"{u}nchen\\
\verb+warzel@ma.tum.de+
\end{minipage}
\end{document}